\documentclass[11pt]{article}
\usepackage{enumerate}
\usepackage{amssymb,a4wide,latexsym,makeidx,epsfig}
\usepackage{amsthm}
\usepackage{dsfont} %\mathds font
\usepackage{amsmath}
\usepackage{lipsum} % for filler text
\usepackage{enumerate}
\usepackage{mathrsfs}
\usepackage{xcolor}
\usepackage{tikz}%%%%draw graph%%%PDFLaTeX or PDFTeXify
\usepackage{geometry}
\usepackage{appendix}
\usepackage{multirow} %合并行
\usepackage[colorlinks, linkcolor=black, anchorcolor=black, citecolor=black]{hyperref}%参考文献等加链接。colorlinks表示采用颜色作为标识，若不用则采用默认的方框方式。citecolor选项表示颜色。
\usepackage{microtype} %allow LATEX to adjust spacing between characters to reduce the number of bad line breaks.
\usepackage{colonequals} % := use \colonequals
\newtheorem{theorem}{Theorem}[section]

\newtheorem{lemma}[theorem]{Lemma}
\newtheorem{fact}{Fact}[section]

\newtheorem{conjecture}[theorem]{Conjecture}
\newtheorem{claim}{Claim}[section]

\allowdisplaybreaks
\begin{document}
\textwidth 150mm \textheight 225mm
\title{Spectral radius conditions for the existence \\ of all subtrees of diameter at most four
\thanks{Supported by the National Natural Science Foundation of China (No. 11871398)
and China Scholarship Council (No. 202006290071).}
}
\author{{Xiangxiang Liu$^{1,2}$, Hajo Broersma$^{2,}$\thanks{Corresponding author.}, Ligong Wang$^{1}$}\\
{\small $^{1}$ School of Mathematics and Statistics,}\\ {\small Northwestern Polytechnical University, Xi'an, Shaanxi 710129, PR China}\\
{\small $^{2}$ Faculty of Electrical Engineering, Mathematics and Computer Science,}\\ {\small University of Twente, P.O. Box 217, 7500 AE Enschede, The Netherlands}\\
{\small E-mail: xxliumath@163.com; h.j.broersma@utwente.nl; lgwangmath@163.com}}
\date{}
\maketitle
\begin{center}
\begin{minipage}{120mm}
\vskip 0.3cm
\begin{center}
{\small {\bf Abstract}}
\end{center}
{\small
Let $\mu(G)$ denote the spectral radius of a graph $G$. We partly confirm a conjecture due to Nikiforov, which is a spectral radius analogue of the well-known Erd\H{o}s-S\'os Conjecture that any tree of order $t$ is contained in a graph of average degree greater than $t-2$. Let $S_{n,k}=K_{k}\vee\overline{K_{n-k}}$, and let $S_{n,k}^{+}$ be the graph obtained from $S_{n,k}$ by adding a single edge joining two vertices of the independent set of $S_{n,k}$. In 2010, Nikiforov conjectured that for a given integer $k$, every graph $G$ of sufficiently large order $n$ with $\mu(G)\geq \mu(S_{n,k}^{+})$ contains all trees of order $2k+3$, unless $G=S_{n,k}^{+}$. We confirm this conjecture for trees with diameter at most four, with one exception. In fact, we prove the following stronger result for $k\geq 8$. If a graph $G$ with sufficiently large order $n$ satisfies $\mu(G)\geq \mu(S_{n,k})$ and $G\neq S_{n,k}$, then $G$ contains all trees of order $2k+3$ with diameter at most four, except for the tree obtained from a star $K_{1,k+1}$ by subdividing each of its $k+1$ edges once.
}
{\small
}
\end{minipage}
\end{center}

\section{Introduction}
\label{sec:trees-intro}
The work presented here is motivated by several open problems (and partial solutions) in extremal graph theory and their counterparts in spectral graph theory. In extremal graph theory, a well-known Tur\'an-type problem is to determine the maximum number of edges a graph can have without admitting a specific subgraph. Similarly, the celebrated Erd\H{o}s-S\'os Conjecture comes down to determining the maximum number of edges a graph can have without admitting some tree of a specific order. Such problems and conjectures have natural counterparts in spectral graph theory, e.g., dealing with the maximum spectral radius a graph can have without belonging to a specific class of graphs, or without containing a specific subgraph. The latter type of problem was coined a Brualdi-Solheid-Tur\'{a}n type problem by Nikiforov in \cite{Nikiforov}. There the author also argues that  ``the study of Brualdi-Solheid-Tur\'an type problems is an important topic in spectral graph theory''. Before we give more details about the background and motivation for our work, we start by giving some essential definitions and introducing some useful notation.

Let $G$ be a simple undirected graph with vertex set $V(G)=\{v_{1},v_{2},\ldots,v_{n}\}$ and edge set $E(G)$, where $n$ is called the order of $G$. Let $e(G)=|E(G)|$. The \emph{adjacency matrix} $A(G)$ of $G$ is the $n\times n$ matrix $(a_{ij})$, where $a_{ij}=1$ if $v_{i}$ is adjacent to $v_{j}$, and $a_{ij}=0$ otherwise. The \emph{spectral radius} of $G$ is the largest eigenvalue of $A(G)$, denoted by $\mu(G)$. The spectral radius of a graph $G$ is related to the distribution of edges in $G$. For example, it is well-known that the average degree of $G$ is at most $\mu(G)$, and $\mu(G)$ is at most the maximum degree of $G$.

Adopting the notation of \cite{Nikiforov}, let $S_{n,k}$ denote the graph obtained by joining every vertex of a complete graph $K_{k}$ to every vertex of an independent set of order $n-k$, that is, $S_{n,k}=K_{k}\vee\overline{K_{n-k}}$. Let $S_{n,k}^{+}$ be the graph obtained from $S_{n,k}$ by adding a single edge joining two vertices of the independent set of $S_{n,k}$. A \emph{star} of order $t+1$ is denoted by $K_{1,t}$. A star $K_{1,t}$ is called \emph{nontrivial} if $t\geq 1$. For convenience, we consider $K_{1,0}=K_1$ as a trivial star. Let $S_{2,\ldots,2}$ be the tree of order $2k+3$ obtained from a star $K_{1,k+1}$ by subdividing each of its $k+1$ edges exactly once (see Figure~\ref{trees1}).

A graph $G$ is said to be \emph{$H$-free}, if $H$ is not a subgraph of $G$. In order to avoid confusion, please note that we mean subgraph here and not induced subgraph. The \emph{Tur\'{a}n number} of $H$ is the maximum number of edges in an $H$-free graph of order $n$, and denoted by $ex(n,H)$. As a natural spectral radius counterpart of Tur\'{a}n type problems, Nikiforov \cite{Nikiforov} posed the following general Brualdi-Solheid-Tur\'{a}n type problem: what is the maximum spectral radius of an $H$-free graph of order $n$? In the past decade, this problem has been studied for various choices of $H$, such as for complete graphs \cite{Wilf}, complete bipartite graphs \cite{Babai,Nikiforov1}, and for cycles or paths of specified length \cite{GaoHou,Nikiforov2,Nikiforov,Zhailin,ZhaiWang}. More results on spectral extremal problems can be found in \cite{ChenLiu,ChenLiu1,Ciofeng,Tait}.

As we mentioned before, our work is motivated by a spectral radius analogue of the well-known Erd\H{o}s-S\'{o}s Conjecture that a graph of average degree greater than $t-2$ admits any tree of order $t$. This conjecture was solved for trees of diameter at most four by McLennan \cite{McLenn}. Based on the Erd\H{o}s-S\'{o}s Conjecture, Nikiforov proposed the following Brualdi-Solheid-Tur\'{a}n type conjecture concerning trees.

\begin{conjecture}\label{co:trees1}{\normalfont (\cite{Nikiforov})} Let $k\geq 2$ and let $G$ be a graph of sufficiently large order $n$.
\begin{itemize}
\item[{\rm (a)}] If $\mu(G)\geq \mu(S_{n,k})$, then $G$ contains all trees of order $2k+2$, unless $G=S_{n,k}$.
\item[{\rm (b)}] If $\mu(G)\geq \mu(S_{n,k}^{+})$, then $G$ contains all trees of order $2k+3$, unless $G=S_{n,k}^{+}$.
\end{itemize}
\end{conjecture}

In \cite{Nikiforov}, Nikiforov confirmed both conjectures for the paths $P_{2k+2}$ and $P_{2k+3}$, respectively. Recently, Hou et al. \cite{HouLiu} proved that Conjecture~\ref{co:trees1} (a) holds for all trees of diameter at most four. In this paper, we show that Conjecture~\ref{co:trees1} (b) holds for all trees of diameter at most four, except for $S_{2,\ldots,2}$. Note that $\mu(S_{n,k})<\mu(S_{n,k}^+)$. We in fact prove the following stronger result.

\begin{theorem}\label{th:trees}
Let $k\geq 8$ and let $G$ be a graph of sufficiently large order $n$. Let $\mathcal{T}$ be the set of all trees of order $2k+3$ and diameter at most four, except for $S_{2,\ldots,2}$. If $\mu(G)\geq \mu(S_{n,k})$ and $G\neq S_{n,k}$, then $G$ contains all trees in $\mathcal{T}$.
\end{theorem}

The remainder of this paper is organized as follows. In the next section, we provide some additional terminology and auxiliary results that will be used in our proofs. In Sections~\ref{sec:trees results1} and~\ref{sec:trees2}, we state and prove three theorems which together imply Theorem~\ref{th:trees}.
%%%%%%%%%%%%%%%%%%%%%%%%%%%%%%%%%%%%%%%%%%%%%%%%%%%%%%%%%%%%%%%%%%%%%%%%%%%%%%%%%%%
%%%%%%%%%%%%%%%%%%%%%%%%%%%%%%%%%%%%%%%%%%%%%%%%%%%%%%%%%%%%%%%%%%%%%%%%%%%%%%%%%%%

\section{Preliminaries}\label{sec:trees-pre}

We begin with some additional terminology and notation. Let $G=(V(G),E(G))$ be a graph. For $u\in V(G)$, let $N^{d}(u)=\{v\in V(G)\colon\, d_{G}(v,u)=d\}$, where $d_{G}(v,u)$ is the distance between $u$ and $v$ in $G$. Let $d_{G}(u)$ be the degree of $u$ in $G$. We write $\omega(G)$ for the number of components of $G$. For a non-empty subset $U\subseteq V(G)$, let $G[U]$ be the subgraph of $G$ induced by $U$, $E(U)$ be the edge set of $G[U]$, and $e(U)=|E(U)|$. For two disjoint vertex sets $U, V\subseteq V(G)$, let $E(U,V)$ be the set of edges in $G$ with one end-vertex in $U$ and one end-vertex in $V$, and let $e(U,V)=|E(U,V)|$. Given two graphs $G$ and $H$ with $H\subseteq G$, let $G\setminus H$ be the subgraph of $G$ induced by $V(G)\setminus V(H)$.

\begin{figure}[htbp]
\begin{centering}
\includegraphics[scale=0.8]{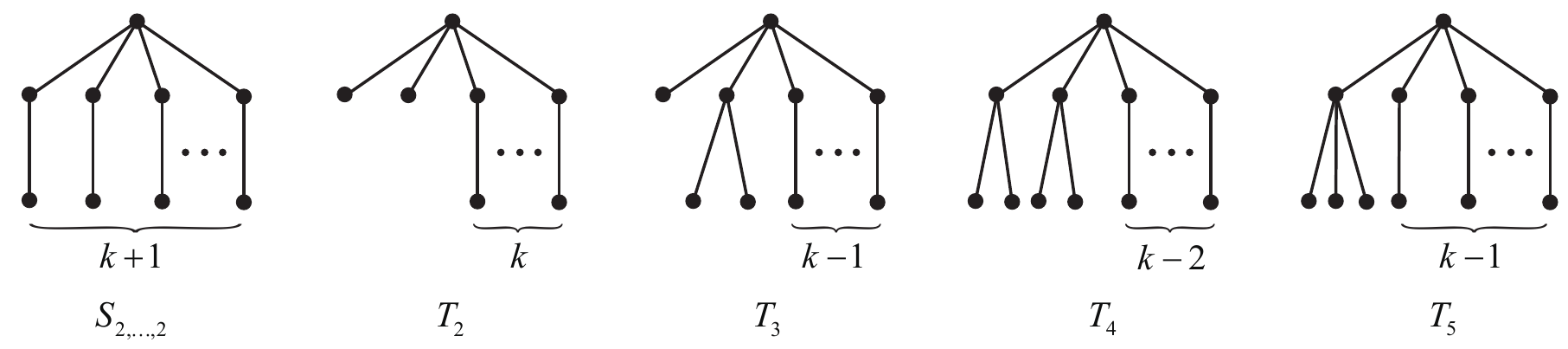}
\caption{The trees $S_{2,\ldots,2}, T_{2}, T_{3}, T_{4}$ and $T_{5}$ of order $2k+3$}
\label{trees1}
\end{centering}
\end{figure}

Let $\mathcal{T}_{2k+3}^{4}$ denote the set of all trees of order $2k+3$ and diameter at most four. It is well-known that a tree has exactly one center or exactly two adjacent centers. Recall that we consider $K_{1,0}=K_1$ as a trivial star. For any $T\in \mathcal{T}_{2k+3}^{4}$, the deletion of a center of $T$ results in a forest each component of which is a star. In the sequel, a tree $T\in \mathcal{T}_{2k+3}^{4}$ is always viewed as a rooted tree with the root at its center. Here we use the following convention. If $T$ has two adjacent centers, then we choose the root at the center $v\in V(T)$  that yields the largest number of components in $T-v$. Let $C_{T}$ (resp., $F_{T}$) be the star forest consisting of all stars (resp., all nontrivial stars) in the subgraph of $T$ obtained by deleting the root of $T$. Let $p\colonequals \omega(C_{T})$ and $p'\colonequals \omega(F_{T})$. Note that $p'\leq k+1$ for any $T\in \mathcal{T}_{2k+3}^{4}$. In particular, $S_{2,\ldots,2}$ is the unique tree in $\mathcal{T}_{2k+3}^{4}$ with $p'=k+1$, and there are exactly four trees $T_2, T_3, T_4, T_5$ in $\mathcal{T}_{2k+3}^{4}$ with $p'=k$ (see Figure~\ref{trees1}). Let $\mathcal{T}^{\ast}=\{T\in \mathcal{T}_{2k+3}^{4}\colon\, p'\leq k-1\}$. Moreover, note that $p+e(F_{T})=2k+2$.

We first prove our main result for the trees in $\mathcal{T}^{\ast}$ in Section~\ref{sec:trees results1}, and next for the four remaining trees in Section~\ref{sec:trees2}. The following lemma on matrices from \cite{GaoHou} is used in the set-up of all our proofs.

Given an $n\times n$ matrix $A$, let $A_{ij}$ be the $(i,j)$-th entry of $A$ for $1\leq i,j\leq n$.

\begin{lemma}\label{le:trees1}{\normalfont (\cite{GaoHou})}
Given $a,b\in \mathbb{Z}^{+}$ and an $n\times n$ nonnegative symmetric irreducible matrix $A$, let $\mu$ be the largest eigenvalue of $A$, and let $\mu'$ be the largest root of $f(x)=x^{2}-ax-b$. Define $B=f(A)=A^{2}-aA-bI$ and let $B_{j}=\sum_{i=1}^{n}B_{ij}$ {\rm(}$1\leq j\leq n${\rm)}. If $B_{j}\leq0$ for all $j\in \{1,2,\ldots,n\}$, then $\mu\leq \mu'$, with equality holding if and only if $B_{j}=0$ for all $j\in \{1,2,\ldots,n\}$.
\end{lemma}

We also frequently use the following known result on the Tur\'an number of a star forest.

\begin{lemma}\label{le:trees2}{\normalfont (\cite{LidiLiu})} Let $F=\bigcup_{i=1}^{t}K_{1,m_{i}}$ be a star forest with $m_{1}\geq m_{2}\geq\cdots\geq m_{t}$. Then for sufficiently large $m$, we have
$$ex(m,F)=\max\limits_{1\leq i\leq t}\left\{(i-1)(m-i+1)+{i-1 \choose 2}+\left\lfloor\frac{m_{i}-1}{2}(m-i+1)\right\rfloor\right\}.$$
\end{lemma}

At some places we apply the following partial solution of the Erd\H{o}s-S\'{o}s Conjecture.

\begin{lemma}\label{le:trees3}{\normalfont (\cite{McLenn})} Every graph $G$ with $e(G)>(k-2)|V(G)|/2$ contains all trees with order $k$ and diameter at most four as subgraphs.
\end{lemma}

We also need the following auxiliary result, in which $F_T$ is defined as above.

\begin{lemma}\label{le:trees4} Let $T\in \mathcal{T}_{2k+3}^{4}$ with $p'\geq2$. Let $H=(X,Y)$ be a bipartite graph with $|X|\geq e(F_{T})$ and $|Y|\geq p'$. If there exist
vertices $y_{0},y_{1},\ldots,y_{p'-1}\in Y$ such that $d_{H}(y_{0})\geq e(F_{T})$ and $d_{H}(y_{i})\geq e(F_{T})-1$ for all $i\in \{1,2,\ldots,p'-1\}$, then $H$ contains a copy of $F_{T}$ such that all the centers of its components are contained in $Y$.
\end{lemma}

\begin{proof} Since $p'\geq2$, each component of $F_{T}$ has at most $e(F_{T})-1$ edges. Assume that $K_{1,m_{1}},K_{1,m_{2}},\ldots,K_{1,m_{p'}}$ are the components of $F_{T}$. We shall first choose disjoint sets $X_{1},X_{2},\ldots,$ $X_{p'-1}$ one-by-one in $X$ such that $y_{i}$ is completely joined to $X_{i}$ and $|X_{i}|=m_{i}$ for each $i\in \{1,2,\ldots,p'-1\}$.

Since $d_{H}(y_{1})\geq e(F_{T})-1\geq m_{1}$, we can find a set $X_{1}\subseteq X$ satisfying the requirement. Suppose that we have found the desired sets $X_{1},\ldots,X_{i-1}$ for some $1\leq i \leq p'-1$. Since $|\bigcup_{j=1}^{i-1}X_{j}|=\sum_{j=1}^{i-1}m_{j}=e(F_{T})-\sum_{j=i}^{p'}m_{j}\leq e(F_{T})-m_{i}-1\leq d_{H}(y_{i})-m_{i}$, we can find $X_{i}\subseteq X\setminus (\bigcup_{j=1}^{i-1}X_{j})$ such that $|X_{i}|=m_{i}$ and $y_{i}$ is completely joined to $X_{i}$. This shows that we can find the desired sets $X_{1},X_{2},\ldots,X_{p'-1}$.

Since $d_{H}(y_{0})\geq e(F_{T})$ and $|X|\geq e(F_{T})$, we can find a set $X_{p'}\subseteq X\setminus (\bigcup_{j=1}^{p'-1}X_{j})$ such that $|X_{p'}|=m_{p'}$ and $y_{0}$ is completely joined to $X_{p'}$. Thus $H$ contains a copy of $F_{T}$ such that all the centers of its components are contained in $Y$.
\end{proof}

We now have all the ingredients to prove our main result. We start with our proof for the trees in the set $\mathcal{T}^{\ast}$, i.e.,
with diameter at most four and $p'\leq k-1$.

%%%%%%%%%%%%%%%%%%%%%%%%%%%%%%%%%%%%%%%%%%%%%%%%%%%%%%%%%%%%%%%%%%%%%%%%%%%%%%%%%%%

\section{Trees with diameter at most four and $p'\leq k-1$}\label{sec:trees results1}
In this section, we are going to prove the following result for all trees in $\mathcal{T}^{\ast}$.

\begin{theorem}\label{th:trees1} For integers $k\geq8$ and $n$ sufficiently large, every graph $G$ of order $n$ with $\mu(G)\geq \mu(S_{n,k})$ contains all trees $T\in \mathcal{T}^{\ast}$.
\end{theorem}

\begin{proof}
For a contradiction, suppose that $G$ is an $n$-vertex graph with $\mu(G)\geq \mu(S_{n,k})$ and there exists a tree $T\in \mathcal{T}^{\ast}$ such that $G$ contains no copy of
$T$. Since $S_{n,k}$ contains all trees in $\mathcal{T}^{\ast}$, we may assume that $G\neq S_{n,k}$. In addition, we have $p'\leq k-1$ and $p+p'\leq p+e(F_{T})=2k+2$.

Let $f(x)=x^{2}-(k-1)x-k(n-k)$. Note that $\mu(S_{n,k})$ is the largest root of $f(x)$. Let $B=f(A(G))$ and $B_{v}=\sum_{i=1}^{n}B_{iv}$ for any $v\in V(G)$. By Lemma~\ref{le:trees1} and since $\mu(G)\geq \mu(S_{n,k})$, there exists a vertex $u\in V(G)$ with $B_{u}\geq 0$. Let $L_{u}$ be the graph with vertex set $N^{1}(u)\cup N^{2}(u)$ and edge set $E(N^{1}(u))\cup E(N^{1}(u),N^{2}(u))$. By the definition of $B$, we have
\begin{equation}\label{eq:trees1}
B_{u}=\sum\limits_{x\in N^{1}(u)}d_{L_{u}}(x)-(k-2)d_{G}(u)-k(n-k).
\end{equation}

We complete the proof by first proving two claims and then distinguishing three cases based on the degree of $u$.

\begin{claim}\label{cl:trees1} $d_{G}(u)\geq k+1$.
\end{claim}

\begin{proof}
Since $d_{L_{u}}(x)\leq n-2$ for all $x\in N^{1}(u)$, by equality (\ref{eq:trees1}), we have
\begin{equation}\label{eq:trees2}
0\leq B_{u}\leq d_{G}(u)(n-2)-(k-2)d_{G}(u)-k(n-k)=(d_{G}(u)-k)(n-k).
\end{equation}
Thus $d_{G}(u)\geq k$. If $d_{G}(u)\geq k+1$, then we are done. If $d_{G}(u)=k$, then inequality (\ref{eq:trees2}) implies $0=B_{u}=d_{G}(u)(n-2)-(k-2)d_{G}(u)-k(n-k)$. In this case, $d_{L_{u}}(x)=n-2$ for all $x\in N^{1}(u)$. Then $G$ contains $S_{n,k}$ as a subgraph, implying that $G$ contains a copy of $T$, a contradiction.
\end{proof}

By Claim~\ref{cl:trees1}, we have $B_{u}\geq0$ and $d_{G}(u)\geq k+1$. Then
\begin{align}\label{eq:trees3}
&\sum_{x\in N^{1}(u)}~\sum_{y\in N^{1}(x)\cap N^{2}(u)}d_{L_{u}}(y)=\sum_{y\in N^{2}(u)}d_{L_{u}}^{2}(y)\nonumber\\
\geq&~\frac{1}{|N^{2}(u)|}\Bigg(\sum_{y\in N^{2}(u)}d_{L_{u}}(y)\Bigg)^{2}\geq\frac{1}{|N^{2}(u)|}\Bigg(\sum_{x\in N^{1}(u)}d_{L_{u}}(x)-|N^{1}(u)|(|N^{1}(u)|-1)\Bigg)^{2}\nonumber\\
>&~\frac{1}{n}\left(B_{u}+(k-2)d_{G}(u)+k(n-k)-d_{G}(u)(d_{G}(u)-1)\right)^{2}\nonumber\\
\geq&~\frac{1}{n}\left(k(n-k)-d_{G}(u)(d_{G}(u)-k+1)\right)^{2}.
\end{align}

\begin{claim}\label{cl:trees2} $p'\geq1$.
\end{claim}

\begin{proof}
Suppose that $p'=0$, that is, $T=K_{1,2k+2}$. If $d_{G}(u)\geq 2k+2$, then $G$ contains a copy of $T$. Thus $d_{G}(u)\leq 2k+1$. Then by inequality (\ref{eq:trees3}), we have $\sum_{x\in N^{1}(u)}\sum_{y\in N^{1}(x)\cap N^{2}(u)}d_{L_{u}}(y)\geq\frac{1}{n}\left(kn-k^{2}-(2k+1)(2k+1-k+1)\right)^{2}=\Omega(n)$.

On the other hand, since $G$ contains no copy of $T$, we have $|N^{1}(x)|\leq2k+1$ for all $x\in N^{1}(u)$. Then $\sum_{x\in N^{1}(u)}\sum_{y\in N^{1}(x)\cap N^{2}(u)}d_{L_{u}}(y)\leq\sum_{x\in N^{1}(u)}\sum_{y\in N^{1}(x)\cap N^{2}(u)}|N^{1}(u)|\leq 2k|N^{1}(u)|^{2}\leq2k(2k+1)^{2}=o(n)$. This contradiction completes the proof of Claim~\ref{cl:trees2}.
\end{proof}

We divide the rest of the proof into three cases: (i) $d_{G}(u)\leq p-1$, (ii) $p\leq d_{G}(u)\leq \frac{n}{k+2}$ and (iii) $d_{G}(u)>\frac{n}{k+2}$.

\medskip\noindent
{\bf Case 1.} $d_{G}(u)\leq p-1$.
\vspace{0.05cm}

For $x\in N^{1}(u)$, let $C(x)=\{y\in N^{1}(x)\cap N^{2}(u)\colon\, d_{L_{u}}(y)\geq2k+2-p\}$. We need the following claim in order to complete this case.

\begin{claim}\label{cl:trees3} There exists a vertex $x\in N^{1}(u)$ with $|C(x)|\geq p'-1$ and $|N^{1}(x)\cap N^{2}(u)|\geq p-1$.
\end{claim}

\begin{proof}
First suppose that $|C(x)|<p'-1$ for every $x\in N^{1}(u)$. Then
\begin{align*}\label{eq:trees3}
\sum_{x\in N^{1}(u)}~\sum_{y\in N^{1}(x)\cap N^{2}(u)}d_{L_{u}}(y)<&~d_{G}(u)((p'-1)d_{G}(u)+(n-1-d_{G}(u)-p'+1)(2k+1-p))\\
=&~d_{G}(u)(n(2k+1-p)-d_{G}(u)(2k-p-p'+2)-p'(2k+1-p))\\
\leq&~d_{G}(u)(2k+1-p)n\leq d_{G}(u)(2k-d_{G}(u))n\leq(k^{2}-1)n.
\end{align*}

On the other hand, by inequality (\ref{eq:trees3}), we have
\begin{align*}
&\sum_{x\in N^{1}(u)}~\sum_{y\in N^{1}(x)\cap N^{2}(u)}d_{L_{u}}(y)\geq\frac{1}{n}\left(k(n-k)-d_{G}(u)(d_{G}(u)-k+1)\right)^{2}\\
\geq&~\frac{1}{n}\left(k(n-k)-(p-1)(p-k)\right)^{2}\geq\frac{1}{n}\left(k(n-k)-2k(k+1)\right)^{2}=\frac{1}{n}\left(kn-3k^{2}-2k\right)^{2}.
\end{align*}
For sufficiently large $n$, we have that $(k^{2}-1)n<\frac{1}{n}\left(kn-3k^{2}-2k\right)^{2}$, contradicting the combination of the above two inequalities.
Thus there exists a vertex $x\in N^{1}(u)$ with $|C(x)|\geq p'-1$.

Next suppose that $|N^{1}(x)\cap N^{2}(u)|\leq p-2$ for every $x\in N^{1}(u)$ with $|C(x)|\geq p'-1$. Then
\begin{align*}
&\sum_{x\in N^{1}(u)}~\sum_{y\in N^{1}(x)\cap N^{2}(u)}d_{L_{u}}(y)\\
=&\sum_{x\in N^{1}(u),|C(x)|\geq p'-1}~\sum_{y\in N^{1}(x)\cap N^{2}(u)}d_{L_{u}}(y)+\sum_{x\in N^{1}(u),|C(x)|\leq p'-2}~\sum_{y\in N^{1}(x)\cap N^{2}(u)}d_{L_{u}}(y)\\
\leq&\sum_{x\in N^{1}(u),|C(x)|\leq p'-2}((p'-2)|N^{1}(u)|+(n-1-|N^{1}(u)|-p'+2)(2k+1-p))\\
&+\sum_{x\in N^{1}(u),|C(x)|\geq p'-1}(p-2)|N^{1}(u)|\\
\leq&\sum_{x\in N^{1}(u),|C(x)|\leq p'-2}(2k+1-p)n+\sum_{x\in N^{1}(u),|C(x)|\geq p'-1}(p-2)(p-1).
\end{align*}

By Claim~\ref{cl:trees2}, we have $p\leq 2k+1$. If $2k+1-p>0$, then
\begin{align*}
&\sum_{x\in N^{1}(u),|C(x)|\leq p'-2}(2k+1-p)n+\sum_{x\in N^{1}(u),|C(x)|\geq p'-1}(p-2)(p-1)\\
\leq~&(p-2)(p-1)+(|N^{1}(u)|-1)(2k-|N^{1}(u)|)n\leq(p-2)(p-1)+(k^{2}-k)n,
\end{align*}
which contradicts $\sum_{x\in N^{1}(u)}\sum_{y\in N^{1}(x)\cap N^{2}(u)}d_{L_{u}}(y)\geq \frac{1}{n}\left(kn-3k^{2}-2k\right)^{2}$.

If $2k+1-p=0$, then
\begin{align*}
&\sum_{x\in N^{1}(u),|C(x)|\leq p'-2}(2k+1-p)n+\sum_{x\in N^{1}(u),|C(x)|\geq p'-1}(p-2)(p-1)\\
\leq~&(p-2)(p-1)|N^{1}(u)|\leq(p-2)(p-1)^{2},
\end{align*}
which contradicts $\sum_{x\in N^{1}(u)}\sum_{y\in N^{1}(x)\cap N^{2}(u)}d_{L_{u}}(y)\geq \frac{1}{n}\left(kn-3k^{2}-2k\right)^{2}$.
\end{proof}

By Claim~\ref{cl:trees3}, there exists a vertex $x\in N^{1}(u)$ with $|C(x)|\geq p'-1$ and $|N^{1}(x)\cap N^{2}(u)|\geq p-1$. Note that for any $y\in C(x)$, $d_{L_{u}-x}(y)\geq 2k+1-p=e(F_{T})-1$. Moreover, since $k+1\leq d_{G}(u)\leq p-1$, we have $p\geq k+2$. Thus $e(F_{T})=2k+2-p\leq k\leq |N^{1}(u)\setminus \{x\}|$.

If $p'\geq2$, then by Lemma~\ref{le:trees4} the graph $(X,Y)$ with $X=N^{1}(u)\setminus \{x\}$ and $Y=\{u\}\cup C(x)$ contains a copy of $F_{T}$. Together with $x$ and $p-p'$ vertices of $(N^{1}(x)\cap N^{2}(u))\setminus Y$ this copy of $F_{T}$ forms a copy of $T$, a contradiction. Thus $p'=1$. Now there exists a copy of $T$ with center $x$ such that $F_{T}$ is a star centered at $u$ and $C_{T}\setminus F_{T}$ consists of $p-1$ vertices in $N^{1}(x)\cap N^{2}(u)$. This contradiction completes the proof for Case 1.

\medskip\noindent
{\bf Case 2.} $p\leq d_{G}(u)\leq \frac{n}{k+2}$.
\vspace{0.05cm}

Let $C=\{x\in N^{1}(u)\colon\, d_{L_{u}}(x)\geq2k+1\}$. We claim that $|C|\geq p'$. Otherwise, if $|C|\leq p'-1\leq k-2$, then by equality (\ref{eq:trees1}), we have
\begin{align*}
B_{u}&=\sum\limits_{x\in N^{1}(u)}d_{L_{u}}(x)-(k-2)d_{G}(u)-k(n-k)\\
&\leq~(k-2)(n-2)+2k(d_{G}(u)-k+2)-(k-2)d_{G}(u)-k(n-k)\\
&=~d_{G}(u)(k+2)-2n-k^{2}+2k+4<0,
\end{align*}
contradicting the fact that $B_{u}\geq0$.

We choose $C'\subseteq C$ with $|C'|=p'$ and $C''\subseteq N^{1}(u)\setminus C'$ with $|C''|=p-p'$. Let $X=(N^{1}(u)\cup N^{2}(u))\setminus (C'\cup C'')$. Note that for any $y\in C'$, we have that $y$ has at least $d_{L_{u}}(y)-|(C'\cup C'')\setminus \{y\}|\geq e(F_{T})$ neighbors in $X$. Let $C'=Y$. Then by Lemma~\ref{le:trees4}, there is a copy of
$F_{T}$ between $X$ and $Y$. Together with $u$ and $C''$ this copy of $F_{T}$ forms a copy of $T$, a contradiction.

\medskip\noindent
{\bf Case 3.} $d_{G}(u)>\frac{n}{k+2}$.
\vspace{0.05cm}

We consider two subcases based on the number of edges between $N^{1}(u)$ and $N^{2}(u)$.

\medskip\noindent
{\bf Subcase 3.1.} $e(N^{1}(u),N^{2}(u))>(k+2)d_{G}(u)+2k|N^{2}(u)|-k(n-k)$.
\vspace{0.05cm}

By equality (\ref{eq:trees1}) and since $B_{u}\geq0$, we have $\sum_{x\in N^{1}(u)}d_{L_{u}}(x)\geq(k-2)d_{G}(u)+k(n-k)$. Then
\begin{align*}
e(L_{u})&=~\frac{1}{2}\Bigg(\sum_{x\in N^{1}(u)}d_{L_{u}}(x)+e(N^{1}(u),N^{2}(u))\Bigg)\\
&>~\frac{1}{2}((k-2)d_{G}(u)+k(n-k)+(k+2)d_{G}(u)+2k|N^{2}(u)|-k(n-k))\\
&=~k(d_{G}(u)+|N^{2}(u)|)=\frac{2k+2-2}{2}|V(L_{u})|.
\end{align*}

By Claim~\ref{cl:trees2}, there is a vertex $w$ in $T$ which is at distance 2 from the center of $T$. Let $w'$ be the neighbor of $w$ in $T$ and let $T'=T-w$. Then $T'$ is a tree of order $2k+2$ and diameter at most four. By Lemma~\ref{le:trees3}, $L_u$ contains a copy of $T'$. If the center of such $T'$ is contained in $N^1(u)$, then we delete $w'$ and all the leaves adjacent to $w'$ in $T'$, and add $u$ and some vertices of $N^1(u)$, so that we get a copy of $T$ in $G$, a contradiction. If the center of such $T'$ is contained in $N^2(u)$, then $w'\in N^1(u)$ by the definition of $L_u$. In this case, we add $u$ and the edge $uw'$ to $T'$. This way we obtain a copy of $T$ in $G$, a contradiction.

\medskip\noindent
{\bf Subcase 3.2.} $e(N^{1}(u),N^{2}(u))\leq(k+2)d_{G}(u)+2k|N^{2}(u)|-k(n-k)$.
\vspace{0.05cm}

Let $G'=G[N^{1}(u)]$. In this case, we have
\begin{align*}
e(G')&=~\frac{1}{2}\Bigg(\sum_{x\in N^{1}(u)}d_{L_{u}}(x)-e(N^{1}(u),N^{2}(u))\Bigg)\\
&\geq~\frac{1}{2}((k-2)d_{G}(u)+k(n-k)-(k+2)d_{G}(u)-2k|N^{2}(u)|+k(n-k))\\
&=~-2d_{G}(u)+k(n-|N^{2}(u)|)-k^{2}\geq -2d_{G}(u)+k(d_{G}(u)+1)-k^{2}\\
&=~(k-2)|V(G')|-k^{2}+k.
\end{align*}
Recall that $F_{T}$ is a star forest, say $F_{T}=\bigcup_{i=1}^{p'}K_{1,m_{i}}$ with $m_{1}\geq m_{2}\geq\cdots\geq m_{p'}$. We first prove two claims on the structure of $F_{T}$ and then distinguish another two subcases.

\begin{claim}\label{cl:trees4} $m_{1}\leq4$.
\end{claim}

\begin{proof}
Suppose $m_{1}\geq5$. Let $T'=T\setminus K_{1,m_{1}}$. Then $T'$ is a tree with at least $2k+3-(m_1+1)\leq 2k-3$ vertices and diameter at most four. By Lemma~\ref{le:trees3} and
since $e(G')\geq (k-2)|V(G')|-k^{2}+k>\frac{2k-5}{2}|V(G')|$, there exists a copy of $T'$ in $G'$. Together with $u$ and $m_{1}$ vertices of $N^{1}(u)$ this copy of $T'$ forms a copy of $T$ in $G$, a contradiction.
\end{proof}

\begin{claim}\label{cl:trees5} $\omega(F_{T})=k-1$.
\end{claim}

\begin{proof}
Suppose $\omega(F_{T})\leq k-2$. By Claim~\ref{cl:trees4}, we may assume that $m_{1}=\cdots=m_{a}=4$, $m_{a+1}=\cdots=m_{a+b}=3$, $m_{a+b+1}=\cdots=m_{a+b+c}=2$ and $m_{a+b+c+1}=\cdots=m_{p'}=1$, where $0\leq a,b,c\leq p'\leq k-2$. Moreover, $0\leq a\leq\lfloor\frac{2k+2}{5}\rfloor$, $0\leq b\leq\lfloor\frac{2k+2}{4}\rfloor$, $0\leq c\leq\lfloor\frac{2k+2}{3}\rfloor$, $0\leq a+b\leq\lfloor\frac{2k+2}{4}\rfloor$ and $0\leq a+b+c\leq\lfloor\frac{2k+2}{3}\rfloor$. By Lemma~\ref{le:trees2}, for sufficiently large $N$, we have $ex(N,F_{T})=\max\{(a-1)(N-a+1)+{a-1 \choose 2}+\lfloor \frac{3}{2}(N-a+1)\rfloor,\ (a+b-1)(N-a-b+1)+{a+b-1 \choose 2}+(N-a-b+1),\ (a+b+c-1)(N-a-b-c+1)+{a+b+c-1 \choose 2}+\lfloor \frac{1}{2}(N-a-b-c+1)\rfloor,\ (p'-1)(N-p'+1)+{p'-1 \choose 2}\}<(k-2)N-k^{2}+k$.

Since $e(G')\geq (k-2)|V(G')|-k^{2}+k$, $G'$ contains a copy of $F_{T}$, which together with $u$ and $p-p'$ vertices of $N^{1}(u)$ forms a copy of $T$, a contradiction.
\end{proof}

By Claims~\ref{cl:trees4} and~\ref{cl:trees5}, we have $\omega(F_{T})=k-1$ and $m_{1}\leq4$. Since $k\geq8$, $F_{T}$ contains at least one copy of $K_{1,1}$. Thus $m_{p'}=1$. Let $F_{T}^{\ast}=F_{T}\setminus K_{1,m_{p'}}=\bigcup_{i=1}^{p'-1}K_{1,m_{i}}$. By Lemma~\ref{le:trees3}, we have $ex(N,F_{T}^{\ast})\leq(p'-2)(N-p'+2)+{p'-2 \choose 2}$. Since $e(G')\geq (k-2)|V(G')|-k^{2}+k$, $G'$ contains a copy of $F_{T}^{\ast}$. Let $x_{1},\ldots,x_{p'-1}$ be the centers of $K_{1,m_{1}},\ldots,K_{1,m_{p'-1}}$ in this $F_{T}^{\ast}$, respectively. Let $Y=V(F_{T}^{\ast})\setminus \{x_{1},\ldots,x_{p'-1}\}$ and $W=N^{1}(u)\setminus V(F_{T}^{\ast})$. For the final part of our proof, we distinguish two subcases: (i) $p=p'$ and (ii) $p>p'$.

\medskip\noindent
{\bf Subcase 3.2.1.} $p=p'$.
\vspace{0.05cm}

We need one more claim for our proof in this subcase.

\begin{claim}\label{cl:trees6} $e(Y,W)\leq k+18$.
\end{claim}

\begin{proof}
For any $y\in Y$, we may assume that $y$ is contained in $K_{1,m_{j}}$ in $F_{T}^{\ast}$. If $2\leq m_{j}\leq4$, then $e(\{y\},W)\leq 3$. Otherwise there is a $K_{1,m_{j}}$ between $y$ and $W$, and there is a $K_{1,1}$ between $x_{j}$ and $Y$, which implies that $G'$ contains a copy of $F_{T}$ and thus $G$ contains a copy of $T$.

If $m_{j}=1$, then we may assume that  $e(\{x_{j}\},W)\geq e(\{y\},W)$ without loss of generality. Then $e(\{y\},W)\leq 1$. Otherwise there is a $2K_{1,1}$ between $\{x_{j},y\}$ and $W$, which implies that $G'$ contains a copy of $F_{T}$ and thus $G$ contains a copy of $T$.

From the above arguments, we obtain that $e(Y,W)\leq k+18$.
\end{proof}

In order to avoid a copy of $T$ in $G$, there is no copy of $F_{T}$ in $G'$. Thus there is no edge within $W$. Moreover, for any $w\in W$, we have $e(\{w\},\{x_{1},\ldots,x_{p'-1}\})\leq p'-2$. Otherwise there exists a copy of $T$ centered at $w$ consisting of $F_{T}^{\ast}$, the edge $wu$, and an edge between $u$ and $N^{1}(u)$. Thus
\begin{align*}
e(G')&=~e(G[Y\cup\{x_{1},\ldots,x_{p'-1}\}])+e(G[W])+e(Y,W)+e(\{x_{1},\ldots,x_{p'-1}\},W)\\
&\leq~{2k \choose 2}+k+18+(p'-2)|W|\leq 2k^{2}+18+(k-3)(|V(G')|-2k)\\
&=~(k-3)|V(G')|+6k+18,
\end{align*}
contradicting $e(G')\geq (k-2)|V(G')|-k^{2}+k$.

\medskip\noindent
{\bf Subcase 3.2.2.} $p>p'$.
\vspace{0.05cm}

In this case, we have $2k-4\leq |X\cup Y|\leq2k-1$ and $|N^{1}(u)|-2k+1\leq |W|\leq |N^{1}(u)|-2k+4$. %Similar to the
Similarly as in the proof of Claim~\ref{cl:trees6}, we can deduce that $e(Y,W)\leq k+9$. Moreover, there is no edge within $W$.

We claim that there exist at least $\frac{|W|}{2}$ vertices in $W$ such that each vertex is completely joined to $X$. Indeed, if not, then
\begin{align*}
e(G')&\leq~{|X\cup Y| \choose 2}+e(G[W])+e(Y,W)+e(X,W)\\
&\leq~{2k-1 \choose 2}+k+9+\left(\frac{|W|}{2}-1\right)|X|+\left(\frac{|W|}{2}+1\right)(|X|-1)\\
&=~|W||X|+2k^{2}-2k+9-\frac{|W|}{2}\leq (k-2)|W|+2k^{2}-2k+9-\frac{|N^{1}(u)|-2k+1}{2}\\
&\leq~(k-2)(|N^{1}(u)|-2k+4)+2k^{2}-2k+9-\frac{|N^{1}(u)|-2k+1}{2},
\end{align*}
contradicting $e(G')\geq (k-2)|V(G')|-k^{2}+k$. Thus there exists a subset $W'\subseteq W$ with $|W'|=\frac{|W|}{2}$ which is completely joined to $X$. This implies that there exists a copy of $F_{T}^{\ast}$ between $X$ and $W'$ with centers $x_{1},\ldots,x_{p'-1}$. In order to avoid a copy of $F_{T}$ with centers in $N^{1}(u)$, we have $e(G[Y])=0$, $e(Y,W)=0$ and $e(Y\cup W, V(G)\setminus (N^{1}(u)\cup\{u\}))=0$. Thus all edges within $L_{u}$ have an end-vertex in $X$. We need the following final claim in order to complete our proof of the theorem.

\begin{claim}\label{cl:trees7} The following statements hold.
\begin{itemize}
\item[{\rm (i)}] $|N^{1}(u)\cup N^{2}(u)|>n-1-\frac{k+1}{2}$.
\item[{\rm (ii)}] $e(N^{1}(u),N^{2}(u))>(k+2)d_{G}(u)+2k|N^{2}(u)|-k(n-k)-k^{2}-k$.
\end{itemize}
\end{claim}

\begin{proof}
If one of the above statements does not hold, then $e(G')=\frac{1}{2}(\sum_{x\in N^{1}(u)}d_{L_{u}}(x)-e(N^{1}(u),N^{2}(u)))$ $\geq (k-2)|V(G')|-\frac{k^{2}}{2}+\frac{3k}{2}$. However, $e(G')\leq {|X| \choose 2}+|X|(N^{1}(u)-|X|)=(k-2)|V(G')|-\frac{k^{2}}{2}+\frac{3k}{2}-1$, a contradiction.
\end{proof}

Using Claim~\ref{cl:trees7} (i) to obtain the strict inequality in the fifth step below, we have
\begin{align*}
&~e(N^{1}(u),N^{2}(u))\leq|X||N^{2}(u)|=(k-2)|N^{2}(u)|\\
=&~(k-2)|N^{2}(u)|+(k+2)|N^{2}(u)|+(k+2)d_{G}(u)-k(n-k)-k^{2}-k\\
&-(k+2)|N^{2}(u)|-(k+2)d_{G}(u)+k(n-k)+k^{2}+k\\
=&~2k|N^{2}(u)|+(k+2)d_{G}(u)-k(n-k)-k^{2}-k-(k+2)(|N^{2}(u)|+d_{G}(u))+kn+k\\
<&~2k|N^{2}(u)|+(k+2)d_{G}(u)-k(n-k)-k^{2}-k-(k+2)\left(n-1-\frac{k+1}{2}\right)+kn+k\\
=&~2k|N^{2}(u)|+(k+2)d_{G}(u)-k(n-k)-k^{2}-k-2n+\frac{k^{2}}{2}+\frac{7k}{2}+3,
\end{align*}
contradicting Claim~\ref{cl:trees7} (ii). This completes the proof of Theorem~\ref{th:trees1}.
\end{proof}

%%%%%%%%%%%%%%%%%%%%%%%%%%%%%%%%%%%%%%%%%%%%%%%%%%%%%%%%%%%%%%%%%%%%%%%%%%%%%%%%%%%

\section{Trees with diameter at most four and $p'=k$}\label{sec:trees2}

In this section, we shall prove Theorem~\ref{th:trees} for all trees with diameter at most four and $p'=k$. Recall that there are exactly four trees $T_2, T_3, T_4, T_5$ of order $2k+3$ with diameter at most four and $p'=k$ (see Figure~\ref{trees1}). We first consider the trees $T_4$ and $T_5$.

\begin{theorem}\label{th:trees2} For integers $k\geq4$ and $n$ sufficiently large, every graph $G$ of order $n$ with $\mu(G)\geq \mu(S_{n,k})$ contains all trees $T\in \{T_{4},T_{5}\}$.
\end{theorem}

\begin{proof}
We use the same set-up and notation as in the first paragraphs of the proof of Theorem~\ref{th:trees1}. Similarly as in the proof of Theorem~\ref{th:trees1}, we suppose that $G$ is an $n$-vertex graph with $\mu(G)\geq \mu(S_{n,k})$ and that $G$ contains no copy of $T$ for some $T\in \{T_{4},T_{5}\}$. Since $S_{n,k}$ contains all trees in $\{T_{4},T_{5}\}$, we may assume that $G\neq S_{n,k}$. In addition, we have $p=p'=k$ and $e(F_{T})=k+2$. Like in the proof of Theorem~\ref{th:trees1}, we can also find a vertex $u\in V(G)$ with $B_{u}\geq0$, $d_{G}(u)\geq k+1$, and define $L_{u}$. Now we have $d_{G}(u)\geq k+1=p+1$. Next we consider two cases: (i) $p+1\leq d_{G}(u)\leq\frac{n}{k+2}$ and (ii) $d_{G}(u)>\frac{n}{k+2}$.

\medskip\noindent
{\bf Case 1.} $p+1\leq d_{G}(u)\leq \frac{n}{k+2}$.
\vspace{0.05cm}

Let $C=\{x\in N^{1}(u)\colon\, d_{L_{u}}(x)\geq2k+1\}$. We claim that $|C|\geq k$. Otherwise, if $|C|\leq k-1$, then by equality (\ref{eq:trees1}), we have
\begin{align*}
B_{u}&=~\sum\limits_{x\in N^{1}(u)}d_{L_{u}}(x)-(k-2)d_{G}(u)-k(n-k)\\
&\leq~(k-1)(n-2)+2k(d_{G}(u)-k+1)-(k-2)d_{G}(u)-k(n-k)\\
&=~d_{G}(u)(k+2)-n-k^{2}+2<0,
\end{align*}
contradicting the fact that $B_{u}\geq0$.

We choose $Y\subseteq C$ with $|Y|=k$. Let $X=(N^{1}(u)\cup N^{2}(u))\setminus Y$. Note that for any $y\in Y$, we have that $y$ has at least $2k+1-(k-1)=k+2=e(F_{T})$ neighbors in $X$. Then by Lemma~\ref{le:trees4}, there is a copy of $F_{T}$ between $X$ and $Y$. Together with $u$ this copy of $F_{T}$ forms a copy of $T$, a contradiction.

\medskip\noindent
{\bf Case 2.} $d_{G}(u)>\frac{n}{k+2}$.
\vspace{0.05cm}

We consider three subcases based on the number of edges between $N^{1}(u)$ and $N^{2}(u)$.

\medskip\noindent
{\bf Subcase 2.1.} $e(N^{1}(u),N^{2}(u))>(k+2)d_{G}(u)+2k|N^{2}(u)|-k(n-k)$.
\vspace{0.05cm}

Similarly as in the proof of Subcase~3.1 of Theorem~\ref{th:trees1}, we have $e(L_{u})>\frac{2k+2-2}{2}|V(L_{u})|$, and we conclude that there exists a copy of $T$ in $G$, a contradiction.

\medskip\noindent
{\bf Subcase 2.2.} $e(N^{1}(u),N^{2}(u))<(k+2)d_{G}(u)+2k|N^{2}(u)|-k(n-k)-2k^{2}-6k-10$.
\vspace{0.05cm}

Let $G'=G[N^{1}(u)]$. In this case, we have
\begin{align*}
e(G')&=~\frac{1}{2}\Bigg(\sum_{x\in N^{1}(u)}d_{L_{u}}(x)-e(N^{1}(u),N^{2}(u))\Bigg)\\
&>~\frac{1}{2}((k-2)d_{G}(u)+k(n-k)-(k+2)d_{G}(u)-2k|N^{2}(u)|+k(n-k)+2k^{2}+6k+10)\\
&=~-2d_{G}(u)+k(n-|N^{2}(u)|)+3k+5\geq -2d_{G}(u)+k(d_{G}(u)+1)+3k+5\\
&=~(k-2)|V(G')|+4k+5.
\end{align*}

Note that $F_{T_{4}}=2K_{1,2}\cup(k-2)K_{1,1}$ and $F_{T_{5}}=K_{1,3}\cup(k-1)K_{1,1}$. Let $F_{T_{4}}^{\ast}=2K_{1,2}\cup(k-3)K_{1,1}$ and $F_{T_{5}}^{\ast}=K_{1,3}\cup(k-2)K_{1,1}$. By Lemma~\ref{le:trees2}, we have $e(N,F_{T}^{\ast})\leq(k-2)(N-k+2)+{k-2 \choose 2}=(k-2)N-\frac{k^{2}}{2}+\frac{3k}{2}-1$. Thus $G'$ contains a copy of $F_{T}^{\ast}$. Similarly  as in the proof of Claim~\ref{cl:trees6}, we can deduce that $e(Y,W)\leq2\cdot 4+k-3=k+5$. Moreover, there are no edges within $W$ and $e(X,W)\leq(|V(G')|-2k)(k-2)$. Thus $e(G')\leq {2k \choose 2}+k+5+(|V(G')|-2k)(k-2)=(k-2)|V(G')|+4k+5$, a contradiction.

\medskip\noindent
{\bf Subcase 2.3.} $(k+2)d_{G}(u)+2k|N^{2}(u)|-k(n-k)-2k^{2}-6k-10\leq e(N^{1}(u),N^{2}(u))\leq (k+2)d_{G}(u)+2k|N^{2}(u)|-k(n-k)$.
\vspace{0.05cm}

We first claim that $|V(G)\setminus (N^{1}(u)\cup N^{2}(u)\cup\{u\})|\leq k+5$. Otherwise,
\begin{align*}
e(G')&=~\frac{1}{2}\Bigg(\sum_{x\in N^{1}(u)}d_{L_{u}}(x)-e(N^{1}(u),N^{2}(u))\Bigg)\\
&\geq~\frac{1}{2}((k-2)d_{G}(u)+k(n-k)-(k+2)d_{G}(u)-2k|N^{2}(u)|+k(n-k))\\
&=~-2d_{G}(u)+k(n-|N^{2}(u)|)-k^{2}> -2d_{G}(u)+k(d_{G}(u)+1+k+6)-k^{2}\\
&=~(k-2)|V(G')|+7k>(k-2)|V(G')|+4k+5.
\end{align*}
In this case, we can derive a contradiction by analogous arguments as in Subcase 2.2. Using the claim, we have
\begin{align*}
e(N^{1}(u),N^{2}(u))&\geq~(k+2)d_{G}(u)+2k|N^{2}(u)|-k(n-k)-2k^{2}-6k-10\\
&=~k(|N^{1}(u)|+|N^{2}(u)|)+2|N^{1}(u)|+k|N^{2}(u)|-kn-k^{2}-6k-10\\
&\geq~k(n-1-k-5)+2|N^{1}(u)|+k|N^{2}(u)|-kn-k^{2}-6k-10\\
&=~2|N^{1}(u)|+k|N^{2}(u)|-2k^{2}-12k-10.
\end{align*}

Moreover, $e(G')\geq -2d_{G}(u)+k(n-|N^{2}(u)|)-k^{2}\geq -2d_{G}(u)+k(|N^{1}(u)|+1)-k^{2}=(k-2)|N^{1}(u)|-k^{2}+k$. Let $F_{T_{4}}^{\ast\ast}=2K_{1,2}\cup(k-4)K_{1,1}$ and $F_{T_{5}}^{\ast\ast}=K_{1,3}\cup(k-3)K_{1,1}$. By Lemma~\ref{le:trees2}, we have that $G'$ contains a copy of $F_{T}^{\ast\ast}$. For this copy, let $X=\{x_{1},\ldots,x_{k-2}\}$, $Y=V(F_{T}^{\ast\ast})\setminus X$ and $W=N^{1}(u)\setminus (X\cup Y)$. Without loss of generality, if $x_{i}y$ is an edge of $K_{1,m_{i}}$ for some $y\in Y$ and $m_{i}=1$, then we assume that $e(\{x_{i}\},N^{2}(u))\geq e(\{y\},N^{2}(u))$. In order to avoid a $T$ in $G$, there is no $F_{T}$ in $G[N^{1}(u)\cup N^{2}(u)]$ whose centres are all contained in $N^{1}(u)$. We can use this to prove the following facts.

\begin{fact}\label{fa:trees1} The following statements hold.

\begin{itemize}
\item[{\rm (i)}] $e(W,N^{2}(u))\leq \max\{|W|,|N^{2}(u)|\}$.
\item[{\rm (ii)}] If $e(W,N^{2}(u))\neq0$, then $e(Y,N^{2}(u))\leq 2k+4$.
\item[{\rm (iii)}] If $|N^{2}(u)|\geq7k$, then $e(Y,N^{2}(u))\leq 2|N^{2}(u)|+k-2$.
\end{itemize}
\end{fact}

\begin{proof}
(i) If $e(W,N^{2}(u))\geq \max\{|W|,|N^{2}(u)|\}+1$, then there exists a $2K_{1,1}$ between $W$ and $N^{2}(u)$. Thus there is an $F_{T}$ whose centers are all contained in $N^{1}(u)$, a contradiction.

(ii) If $e(W,N^{2}(u))\geq1$, then we may assume that $wz\in E(G)$ with $w\in W$ and $z\in N^{2}(u)$. For any $y\in Y$, if $y$ is contained in $K_{1,m_{i}}$ of $F_{T}^{\ast\ast}$ with $m_{i}>1$, then $e(\{y\},N^{2}(u))\leq3$. Otherwise there exists an $F_{T}$ with centers $y,w,x_{1},\ldots,x_{k-2}$, a contradiction. Similarly, for any $y\in Y$, if $y$ is contained in $K_{1,1}$ of $F_{T}^{\ast\ast}$, then $e(\{y\},N^{2}(u))\leq2$. Thus $e(Y,N^{2}(u))\leq4\cdot 3+2(k-4)=2k+4$.

(iii) We first consider the case $T=T_{4}$. We may label the vertices in $Y$ using $a,b,c,d,y_{3},\ldots,$ $y_{k-2}$ such that $x_{1}a$, $x_{1}b$, $x_{2}c$, $x_{2}d$, $x_{3}y_{3}$, \ldots, $x_{k-2}y_{k-2}\in E(F_{T_{4}}^{\ast\ast})$. If $e(\{a\},N^{2}(u))\geq4$, then $e(\{c\},N^{2}(u))\leq1$, $e(\{d\},N^{2}(u))\leq1$ and $e(\{y_{i}\},N^{2}(u))\leq1$ for all $3\leq i\leq k-2$. By symmetry, similar statements also hold for $b$, $c$ and $d$. Moreover, for any $i\in \{3,\ldots,k-2\}$ , if $e(\{y_{i}\},N^{2}(u))\geq4$, then $e(\{y_{j}\},N^{2}(u))\leq1$ for $j\in \{3,\ldots,k-2\}\setminus \{i\}$. Thus $e(Y,N^{2}(u))\leq \max\{2|N^{2}(u)|+k-2,\ 3\cdot 4+|N^{2}(u)|+k-5,\ 3\cdot 4+3(k-4)\}= 2|N^{2}(u)|+k-2$.

For the case $T=T_{5}$ we can deduce $e(Y,N^{2}(u))\leq \max\{|N^{2}(u)|+k-1,\ 3\cdot 4+|N^{2}(u)|+k-4,\ 3\cdot 4+3(k-3)\}\leq2|N^{2}(u)|+k-2$.
\end{proof}

Recall that $|V(G)\setminus (N^{1}(u)\cup N^{2}(u)\cup\{u\})|\leq k+5$. Hence, $|W|+|N^{2}(u)|\geq n-1-k-5-(2k-2)=n-3k-4$. Moreover, $e(N^{1}(u),N^{2}(u))\geq 2|N^{1}(u)|+k|N^{2}(u)|-2k^{2}-12k-10$ implies $|N^{2}(u)|\geq2$. We claim that $e(W,N^{2}(u))=0$. Otherwise by Fact~\ref{fa:trees1} (i) and (ii), we have
\begin{align*}
e(X,N^{2}(u))&=e(N^{1}(u),N^{2}(u))-e(W,N^{2}(u))-e(Y,N^{2}(u))\\
&\geq~2|N^{1}(u)|+k|N^{2}(u)|-2k^{2}-12k-10-(|W|+|N^{2}(u)|)-2k-4\\
&=~(k-2)|N^{2}(u)|+|N^{2}(u)|+|W|+2(2k-2)-2k^{2}-14k-14\\
&>~(k-2)|N^{2}(u)|=|X||N^{2}(u)|,
\end{align*}
which is a contradiction.

Hence $e(W,N^{2}(u))=0$. If $2\leq |N^{2}(u)|<7k$, then $e(N^{1}(u),N^{2}(u))\leq |X\cup Y||N^{2}(u)|<14k(k-1)$, which contradicts $e(N^{1}(u),N^{2}(u))\geq 2|N^{1}(u)|+k|N^{2}(u)|-2k^{2}-12k-10$ for sufficiently large $n$. If $|N^{2}(u)|\geq7k$, then by Fact~\ref{fa:trees1} (iii),  we have
\begin{align*}
&~e(N^{1}(u),N^{2}(u))\leq (k-2)|N^{2}(u)|+2|N^{2}(u)|+k-2=k|N^{2}(u)|+k-2\\
=&~2|N^{1}(u)|+k|N^{2}(u)|-2k^{2}-12k-10-2|N^{1}(u)|+2k^{2}+12k+10+k-2\\
<&~2|N^{1}(u)|+k|N^{2}(u)|-2k^{2}-12k-10,
\end{align*}
which contradicts $e(N^{1}(u),N^{2}(u))\geq 2|N^{1}(u)|+k|N^{2}(u)|-2k^{2}-12k-10$.
This completes the proof of Theorem~\ref{th:trees2}.
\end{proof}

Next, we consider the trees $T_2$ and $T_3$.

\begin{theorem}\label{th:trees3} For integers $k\geq3$ and $n$ sufficiently large, every graph $G$ of order $n$ with $\mu(G)\geq \mu(S_{n,k})$ and $G\neq S_{n,k}$ contains all trees $T\in \{T_{2},T_{3}\}$.
\end{theorem}

\begin{proof}
The proof is modelled along similar lines as our proof of Theorem~\ref{th:trees1}. However, we first show that we may assume that $G$ is connected.

Let $G$ be an $n$-vertex graph with $\mu(G)\geq \mu(S_{n,k})$ and $G\neq S_{n,k}$. If $G$ is disconnected, then $G$ has a component $G'$ with $\mu(G)=\mu(G')$. Let $n'=|V(G')|$. Then $n'>\mu(G')=\mu(G)\geq \mu(S_{n,k})=\Omega(n^{1/2})$, which we assume to be sufficiently large. Moreover, we have $\mu(G')=\mu(G)\geq \mu(S_{n,k})>\mu(S_{n',k})$ and $G'\neq S_{n',k}$. In this case, we can consider $G'$ instead of $G$. Hence, we may assume that $G$ is connected.

For a contradiction, suppose that there exists a tree $T\in \{T_{2},T_{3}\}$ which is not a subgraph of $G$. Note that for $T=T_{2}$ we have $p=k+2$, $p'=k$ and $e(F_{T_{2}})=k$, and for $T=T_{3}$ we have $p=k+1$, $p'=k$ and $e(F_{T_{3}})=k+1$.

Similarly as in the proof of Theorem~\ref{th:trees1}, we can find a vertex $u\in V(G)$ with  $B_{u}\geq0$, and define $L_{u}$. Before we start our case distinction, we prove the same first claim as in the proof of Theorem~\ref{th:trees1}.

\begin{claim}\label{cl:trees8} $d_{G}(u)\geq k+1$.
\end{claim}

\begin{proof}
Since $d_{L_{u}}(x)\leq n-2$ for all $x\in N^{1}(u)$, by equality (\ref{eq:trees1}), we have

\begin{equation}\label{eq:trees7}
0\leq B_{u}\leq d_{G}(u)(n-2)-(k-2)d_{G}(u)-k(n-k)=(d_{G}(u)-k)(n-k).
\end{equation}

Thus $d_{G}(u)\geq k$. If $d_{G}(u)\geq k+1$, then we are done. If $d_{G}(u)=k$, then inequality (\ref{eq:trees7}) implies $0=B_{u}=d_{G}(u)(n-2)-(k-2)d_{G}(u)-k(n-k)$. This implies that $d_{L_{u}}(x)=n-2$ for all $x\in N^{1}(u)$. Then $G$ contains $S_{n,k}$ as a subgraph. Since $G\neq S_{n,k}$, $G$ contains $S_{n,k}^{+}$ as a subgraph. Thus $G$ contains a copy of $T$, which is a contradiction.
\end{proof}

We apply the same case distinction as in the proof of Theorem~\ref{th:trees1}. We divide the rest of the proof into three cases: (i) $d_{G}(u)\leq p-1$, (ii) $p\leq d_{G}(u)\leq \frac{n}{k+2}$ and (iii) $d_{G}(u)>\frac{n}{k+2}$.

\medskip\noindent
{\bf Case 1.} $d_{G}(u)\leq p-1$.
\vspace{0.05cm}

In this case, using Claim~\ref{cl:trees8} and that $p=k+1$ for $T=T_{3}$, we have $T=T_{2}$ and $d_{G}(u)=k+1$. For $x\in N^{1}(u)$, let $C(x)=\{y\in N^{1}(x)\cap N^{2}(u)\colon\, d_{L_{u}}(y)\geq k\}$. We first prove another claim in order to complete the proof for this case.

\begin{claim}\label{cl:trees9} There exists a vertex $x\in N^{1}(u)$ with $|C(x)|\geq k-1$ and $|N^{1}(x)\cap N^{2}(u)|\geq k+1$.
\end{claim}

\begin{proof}
Suppose that $|C(x)|<k-1$ for every $x\in N^{1}(u)$. Then
\begin{align*}\label{eq:trees3}
\sum_{x\in N^{1}(u)}~\sum_{y\in N^{1}(x)\cap N^{2}(u)}d_{L_{u}}(y)\leq&~d_{G}(u)((k-2)d_{G}(u)+(n-1-d_{G}(u)-k+2)(k-1))\\
=&~d_{G}(u)(n(k-1)-d_{G}(u)-(k-1)^{2})\leq(k^{2}-1)n.
\end{align*}

On the other hand, we have
\begin{align*}
&~\sum_{x\in N^{1}(u)}~\sum_{y\in N^{1}(x)\cap N^{2}(u)}d_{L_{u}}(y)=\sum_{y\in N^{2}(u)}d_{L_{u}}^{2}(y)\geq \frac{1}{|N^{2}(u)|}\Bigg(\sum_{y\in N^{2}(u)}d_{L_{u}}(y)\Bigg)^{2}\\
\geq&~\frac{1}{|N^{2}(u)|}\Bigg(\sum_{x\in N^{1}(u)}d_{L_{u}}(x)-|N^{1}(u)|(|N^{1}(u)|-1)\Bigg)^{2}\\
>&~\frac{1}{n}\left((k-2)d_{G}(u)+k(n-k)-d_{G}(u)(d_{G}(u)-1)\right)^{2}=\frac{1}{n}\left(k(n-k)-2(k+1)\right)^{2}.
\end{align*}
For
sufficiently large $n$, we have $(k^{2}-1)n~<~\frac{1}{n}\left(k(n-k)-2(k+1)\right)^{2}$, a contradiction. Thus there exists a vertex $x\in N^{1}(u)$ with $|C(x)|\geq k-1$.

If $|N^{1}(x)\cap N^{2}(u)|\leq k$ for every $x\in N^{1}(u)$ with $|C(x)|\geq k-1$, then
\begin{align*}
&\sum_{x\in N^{1}(u)}~\sum_{y\in N^{1}(x)\cap N^{2}(u)}d_{L_{u}}(y)\\
=&~\sum_{x\in N^{1}(u),|C(x)|\geq k-1}~\sum_{y\in N^{1}(x)\cap N^{2}(u)}d_{L_{u}}(y)+\sum_{x\in N^{1}(u),|C(x)|\leq k-2}~\sum_{y\in N^{1}(x)\cap N^{2}(u)}d_{L_{u}}(y)\\
\leq&~\sum_{x\in N^{1}(u),|C(x)|\leq k-2}((k-2)|N^{1}(u)|+(n-1-|N^{1}(u)|-k+2)(k-1))\\
&+~\sum_{x\in N^{1}(u),|C(x)|\geq k-1}k|N^{1}(u)|\\
<&~\sum_{x\in N^{1}(u),|C(x)|\leq k-2}(k-1)n+\sum_{x\in N^{1}(u),|C(x)|\geq k-1}k(k+1)\leq(k^{2}-k)n+k(k+1),
\end{align*}
which contradicts $\sum_{x\in N^{1}(u)}\sum_{y\in N^{1}(x)\cap N^{2}(u)}d_{L_{u}}(y)\geq \frac{1}{n}\left(k(n-k)-2(k+1)\right)^{2}$. Thus there exists a vertex $x\in N^{1}(u)$ with $|C(x)|\geq k-1$ and $|N^{1}(x)\cap N^{2}(u)|\geq k+1$.
\end{proof}

By Claim~\ref{cl:trees9}, there is a vertex $x\in N^{1}(u)$ with $|C(x)|\geq k-1$ and $|N^{1}(x)\cap N^{2}(u)|\geq k+1$. Let $X=N^{1}(u)\setminus \{x\}$ and $Y=C(x)\cup \{u\}$. By Lemma~\ref{le:trees4}, there exists a copy of $T$ centered at $x$, a contradiction.

\medskip\noindent
{\bf Case 2.} $p\leq d_{G}(u)\leq \frac{n}{k+2}$.
\vspace{0.05cm}

Let $C=\{x\in N^{1}(u)\colon\, d_{L_{u}}(x)\geq2k+1\}$. We claim that $|C|\geq k$. Otherwise, if $|C|\leq k-1$, then by equality (\ref{eq:trees1}), we have
\begin{align*}
B_{u}&=\sum\limits_{x\in N^{1}(u)}d_{L_{u}}(x)-(k-2)d_{G}(u)-k(n-k)\\
&\leq~(k-1)(n-2)+2k(d_{G}(u)-k+1)-(k-2)d_{G}(u)-k(n-k)\\
&=~d_{G}(u)(k+2)-n-k^{2}+2<0,
\end{align*}
contradicting the fact that $B_{u}\geq0$.

We choose $Y\subseteq C$ with $|Y|=k$ and $C'\subseteq N^{1}(u)\setminus Y$ with $|C'|=p-p'$. Let $X=(N^{1}(u)\cup N^{2}(u))\setminus (Y\cup C')$. Note that for any $y\in Y$, we have that $y$ has at least $d_{L_{u}}(y)-|(Y\cup C')\setminus \{y\}|\geq e(F_{T})$ neighbors in $X$. Then by Lemma~\ref{le:trees4}, there exists a copy of $F_{T}$ between $X$ and $Y$. Together with $u$ and $C'$ this copy of $F_{T}$ forms a copy of $T$, a contradiction.

\medskip\noindent
{\bf Case 3.} $d_{G}(u)>\frac{n}{k+2}$.
\vspace{0.05cm}

We consider two subcases based on the number of edges between $N^{1}(u)$ and $N^{2}(u)$.

\medskip\noindent
{\bf Subcase 3.1.} $e(N^{1}(u),N^{2}(u))\geq(k+2)d_{G}(u)+2k|N^{2}(u)|-k(n-k)-3k^{2}-k$.
\vspace{0.05cm}

By equality (\ref{eq:trees1}) and since $B_{u}\geq0$, we have $\sum_{x\in N^{1}(u)}d_{L_{u}}(x)\geq(k-2)d_{G}(u)+k(n-k)$. Then
\begin{align*}
e(L_{u})&=~\frac{1}{2}\Bigg(\sum_{x\in N^{1}(u)}d_{L_{u}}(x)+e(N^{1}(u),N^{2}(u))\Bigg)\\
&\geq~\frac{1}{2}((k-2)d_{G}(u)+k(n-k)+(k+2)d_{G}(u)+2k|N^{2}(u)|-k(n-k)-3k^{2}-k)\\
&=~k(d_{G}(u)+|N^{2}(u)|)-\frac{3k^{2}}{2}-\frac{k}{2}=k|V(L_{u})|-\frac{3k^{2}}{2}-\frac{k}{2}.
\end{align*}

By Lemma~\ref{le:trees2}, we have $ex(N,F_{T})=(k-1)(N-k+1)+{k-1 \choose 2}=(k-1)N-\frac{k^{2}}{2}+\frac{k}{2}$. Thus $L_{u}$ contains a copy of $F_{T}=\bigcup_{i=1}^{p'}K_{1,m_{i}}$ with $m_{1}\geq m_{2}\geq\cdots\geq m_{p'}$. If $T=T_{2}$, then $G$ contains a copy of $T$ centered at $u$, a contradiction. If $T=T_{3}$, then we assume that the center of $K_{1,m_{1}}$ is contained in $N^{2}(u)$. Otherwise there is a copy of $T_{3}$ in $G$. Let $F_{T_{3}}\cap N^{1}(u)=Y$, $F_{T_{3}}\cap N^{2}(u)=X$, $N^{1}(u)\setminus Y=W$ and $N^{2}(u)\setminus X=Z$. In order to avoid an $F_{T_{3}}$ whose centers are contained in $N^{1}(u)$, we have $e(Y,W)=0$, $e(Y,Z)=0$  and $e(\{w\},W\cup Z)\leq1$ for any $w\in W$. Moreover, for any $w\in W$, we have $e(\{w\},X)=0$ if $e(\{w\},W\cup Z)=1$. Otherwise there exists a copy of $F_{T_{3}}$ whose centers are contained in $Y\cup \{w\}$. Furthermore, for any $w\in W$, we have $e(\{w\},X\cup W\cup Z)\leq k-1$. Otherwise, if $e(\{w\},X\cup W\cup Z)\geq k$, then $|X|=k$ and $w$ is completely joined to $X$, which implies that all edges of $F_{T}$ have an end-vertex in $X$. In that case, $X\cup Y\cup\{w,u\}$ contains a copy of $T_{3}$ centered at $w$. Now we have
\begin{align*}
e(L_{u})&=~e(X\cup Y)+e(Y,W)+e(Y,Z)+e(W)+e(W,X\cup Z)\\
&\leq~{2k+1 \choose 2}+0+0+|W|(k-1)={2k+1 \choose 2}+(k-1)(|N^{1}(u)|-|Y|)\\
&=~{2k+1 \choose 2}+k|V(L_{u})|-|N^{1}(u)|-k|N^{2}(u)|-(k-1)|Y|\\
&=~k|V(L_{u})|-\frac{3k^{2}}{2}-\frac{k}{2}+\frac{3k^{2}}{2}+\frac{k}{2}+2k^{2}+k-|N^{1}(u)|-k|N^{2}(u)|-(k-1)|Y|\\
&<~k|V(L_{u})|-\frac{3k^{2}}{2}-\frac{k}{2},
\end{align*}
contradicting $e(L_{u})\geq k|V(L_{u})|-\frac{3k^{2}}{2}-\frac{k}{2}$.

\medskip\noindent
{\bf Subcase 3.2.} $e(N^{1}(u),N^{2}(u))<(k+2)d_{G}(u)+2k|N^{2}(u)|-k(n-k)-3k^{2}-k$.
\vspace{0.05cm}

Let $G'=G[N^{1}(u)]$. Then
\begin{align*}
e(G')&=~\frac{1}{2}\Bigg(\sum_{x\in N^{1}(u)}d_{L_{u}}(x)-e(N^{1}(u),N^{2}(u))\Bigg)\\
&>~\frac{1}{2}((k-2)d_{G}(u)+k(n-k)-(k+2)d_{G}(u)-2k|N^{2}(u)|+k(n-k)+3k^{2}+k)\\
&=~-2d_{G}(u)+k(n-|N^{2}(u)|)+\frac{k^{2}}{2}+\frac{k}{2}\geq(k-2)|N^{1}(u)|+\frac{k^{2}}{2}+\frac{3k}{2}.
\end{align*}

Let $F_{T}^{\ast}=\bigcup_{i=1}^{k-1}K_{1,m_{i}}$ where $m_{1}= m_{2}=\cdots=m_{k-1}=1$ if $T=T_{2}$ and $m_{1}=2$, $m_{2}=\cdots=m_{k-1}=1$ if $T=T_{3}$. By Lemma~\ref{le:trees2}, we have $ex(N,F_{T}^{\ast})\leq(k-2)N-\frac{k^{2}}{2}+\frac{3k}{2}-1$. Hence, $G'$ contains a copy of $F_{T}^{\ast}$. Let $x_{1},x_{2},\ldots,x_{k-1}$ be the centers of this $F_{T}^{\ast}$. Let $X=\{x_{1},x_{2},\ldots,x_{k-1}\}$, $Y=V(F_{T}^{\ast})\setminus X$ and $W=N^{1}(u)\setminus (X\cup Y)$. For any $y\in Y$, if $y$ is contained in $K_{1,m_{i}}$ and $m_{i}=1$, then we suppose that $e(\{x_{i}\},W)\geq e(\{y\},W)$. In order to avoid an $F_{T}$ with all centers in $N^{1}(u)$, we have $e(W)=0$, $e(W,N^{2}(u))=0$ and $e(\{y\},W)\leq1$ for $y\in Y$.

We claim that there exists $Z\subseteq Y\cup W$ with $|Z|\geq k+2$ such that every $z\in Z$ is completely joined to $X$. Otherwise,
\begin{align*}
e(G')&=~e(X, Y\cup W)+e(Y,W)+e(X)+e(Y)+e(W)\\
&\leq~(k+1)|X|+(|Y\cup W|-k-1)(|X|-1)+k+{k-1 \choose 2}+{k \choose 2}\\
&=~(k+1)(k-1)+(|N^{1}(u)|-(k-1)-k-1)(k-2)+k+{k-1 \choose 2}+{k \choose 2}\\
&=~|N^{1}(u)|(k-2)+3k,
\end{align*}
contradicting $e(G')>(k-2)|N^{1}(u)|+\frac{k^{2}}{2}+\frac{3k}{2}$. Now we have $e(Y\cup W)=0$ and $e(Y\cup W,N^{2}(u))=0$. Moreover, if $N^{2}(u)\neq\emptyset$, then for any $v\in N^{2}(u)$, $v$ has no neighbors in $G[V(G)\setminus(\{u\}\cup N^{1}(u))]$. Otherwise, if $vv'\in E(G)$ for some $v'\in V(G)\setminus(\{u\}\cup N^{1}(u))$, then there is a copy of $T$ with center $x_{1}$, and the centers of $C_{T}$ are $u,v$ and some vertices in $Z$. Hence, no matter whether $N^{2}(u)\neq\emptyset$ or $N^{2}(u)=\emptyset$, $G$ is a subgraph of $S_{n,k}$, contradicting $\mu(G)\geq \mu(S_{n,k})$ and $G\neq S_{n,k}$.
\end{proof}

\end{document}